\documentclass[11pt, reqno]{article}

\usepackage[utf8]{inputenc}
\usepackage[T1]{fontenc}

\usepackage{setspace}
\usepackage[margin=1in]{geometry}
\usepackage{parskip}
\usepackage[sc]{mathpazo}
\usepackage[T1]{eulervm}
\usepackage[scaled]{helvet}

\usepackage{amsmath,amsthm, amssymb, amsfonts}
\usepackage{mathtools}
\usepackage[hidelinks]{hyperref}

\usepackage[longnamesfirst]{natbib}
\bibpunct[ ]{(}{)}{,}{a}{}{,}

\numberwithin{equation}{section}

\newtheorem{theorem}{Theorem}
    \newtheorem{lemma}[theorem]{Lemma}
    \newtheorem{proposition}[theorem]{Proposition}
    \newtheorem{corollary}[theorem]{Corollary}

\theoremstyle{definition} 
    \newtheorem{definition}[theorem]{Definition}
    \newtheorem{remark}[theorem]{Remark}

\title{The directed landscape seen from its trees}
\author{Mustazee Rahman \and B\'alint Vir\'ag}
\date{}

\begin{document}
\maketitle

\begin{abstract}
    A basic question about the directed landscape is how much of it can be reconstructed simply by knowing the shapes of its geodesics.
    We prove that the directed landscape can be reconstructed from the shapes of its semi-infinite geodesics. In order to show this result,
    we make use of the Busemann process of the directed landscape together with a novel ``coordinate system" and ``differential distance" to capture the behaviour of semi-infinite geodesics.
\end{abstract}

\section{Introduction}

The directed landscape $$\mathcal{L}: \mathbb{R}^4_{\uparrow} \to \mathbb{R}$$
is a random continuous function defined over $\mathbb{R}^4_{\uparrow} = \{(x,s;y,t) \in \mathbb{R}^4: s < t\}$.
Constructed by \cite{DOV}, $\mathcal{L}$ should be thought of as a random signed metric for the plane.
Much like a metric, it admits notions of length and geodesics. This geometric perspective motivates many questions; see for instance the recent works \cite{BB, BGH, Bhatia22, Bhatia23, Bhatia24, Busani, BSS,Dau23, GZ, RV} and references therein for a broader perspective. We refer the reader to \cite{DOV} Definition 1.4 for the definition and basic properties of the directed landscape.

The allure of universality, too, motivates the study of the directed landscape. It stands as the scaling limit of all integrable last passage percolation models, see \cite{DV}, and is expected to be the universal scaling limit of models in the so called KPZ Universality class, including the most natural random planar metric: first passage percolation. Recently, in remarkable progress, $\mathcal{L}$ been shown to be the scaling limit of the KPZ equation by \cite{Wu},  and of ASEP by \cite{ACH} (see also \cite{QS} and \cite{Virag}).

Here we study a reconstruction problem for the directed landscape. It goes as follows. Let $\gamma$ be the geodesic in the directed landscape from $(0,0)$ to $(0,1)$. Is it possible to determine the length of $\gamma$, $\mathcal{L}(0,0;0,1)$, from knowledge of the function $\gamma$ alone? In other words, does the shape of the geodesic determine its length? (One may guess that the length is some kind of a fractal measure of the image of $\gamma$.)

 \cite{Dau21} has shown that it is possible to reconstruct the directed landscape from knowing the shapes of all geodesics. We prove that it is possible to reconstruct the directed landscape from knowing the shapes of all semi-infinite geodesics. It is an interesting question whether one can determine the directed landscape from a single geodesic tree.  \cite{Bhatia24}  announces an upcoming result that the directed landscape can be determined from a single geodesic tree together with the lengths of all geodesics on the tree. 

\section{Main result}

\subsection{The Geodesic Tree}
A semi-infinite geodesic is a continuous function $\gamma : [s, \infty) \to \mathbb{R}$ such that every finite segment $\gamma \mid_{[s,t]}$ is a geodesic in the directed landscape. Semi-infinite geodesics are parametrized by their starting point $p = (\gamma(s),s)$ and direction $\theta = \lim_{t \to \infty} \gamma(t)/t$. One identifies the geodesic with its graph $\{(\gamma(t),t): t \geq s\}$ lying on the plane.

A point $p = (x,t)$ is an \textbf{interior point} of a semi-infinite geodesic if there is an $s < t$ and a geodesic $\gamma : [s, \infty) \to \mathbb{R}$ such that $p = (\gamma(t),t)$. We say that $p$ lies on the geodesic $\gamma$.

The \textbf{geodesic tree} in direction $\theta$ is the collection of all semi-infinite geodesics in direction $\theta$.
This is an uncountable collection, so let us define the object formally.

Let $((x_n,t_n), n \geq 1)$ be an enumeration of $\mathbb{Q}^2$. Consider the following space $\Gamma$. An element of $\Gamma$ is a sequence
$$ (\gamma_1, \gamma_2, \ldots)$$
where each $\gamma_n : [t_n,\infty) \to \mathbb{R}$ is a continuous function with $\gamma_n(t_n) = x_n$.

Let $d_H$ denote the Hausdorff distance for closed subsets of $\mathbb{R}^2$. Given continious functions $\pi_i:[t_i,\infty) \to \mathbb{R}$, $i = 1,2$, define
$$ d_{path}(\pi_1,\pi_2) = \sum_{N=1}^{\infty} \frac{d_H \left ( \{(\pi_1(s),s): s \geq t_1\} \cap [-N,N]^2, \{(\pi_2(s),s): s \geq t_2\} \cap [-N,N]^2 \right)}{2^N}.$$
Now, given $\vec{\gamma}$ and $\vec{\gamma'}$ in $\Gamma$, define
$$ d_{\Gamma}(\vec{\gamma}, \vec{\gamma'}) = \sum_{n=1}^{\infty} \frac{d_{path}(\gamma_n, \gamma'_n)}{2^n}.$$
One can verify that $d_{\Gamma}$ is a metric on $\Gamma$ and that $\Gamma$ is a Polish space in this metric.
Let $\mathcal{B}_{\Gamma}$ be the associated Borel $\sigma$-algebra.

The directed landscape is a random variable
\begin{equation} \label{eqn:dlrv} \mathcal{L} : (\Omega, \mathcal{F}, \mathbb{P}) \to C(\mathbb{R}^4_{\uparrow}, \mathbb{R})\end{equation}
where $C(\mathbb{R}^4_{\uparrow}, \mathbb{R})$ is the space of continuous function from $\mathbb{R}^4_{\uparrow} \to \mathbb{R}$ in the (metrizable) topology of uniform convergence over compacts.
Let $\mu$ denote the law of $\mathcal{L}$, which is a probability measure on $C(\mathbb{R}^4_{\uparrow}, \mathbb{R})$ with its associated Borel $\sigma$-algebra.
Let $S$ be the (closed) support of $\mu$. If $\mathcal{B}_S$ is the Borel $\sigma$-algebra over $S$, then $(S, \mathcal{B}_S, \mu)$ is a probability space.

The geodesic tree in direction $\theta \in \mathbb{R}$ is a random variable
\begin{equation} \label{eqn:geotreedef} \mathrm{GT}_{\theta}: (S, \mathcal{B}_S, \mu) \to (\Gamma, \mathcal{B}_{\Gamma}).\end{equation}
It assigns to each $(x_n,t_n) \in \mathbb{Q}^2$ the unique $\theta$-directed geodesic $\gamma_n$ of $\mathcal{L}$ from $(x_n,t_n)$.
The reason this is a random variable is because the mapping $s \in S \mapsto \mathrm{GT}_{\theta}(s)$ is continuous.
This follows from the way geodesics are constructed, see \cite[Theorem 3.12]{RV}.

So the geodesic tree $\mathrm{GT}_{\theta}$ is now defined formally. To work with it we need access to the interior points of all $\theta$-directed geodesics.
Suppose $\gamma :[s,\infty) \to \mathbb{R}$ is an arbitrary $\theta$-directed geodesic and $p = (\gamma(t), t)$ for $t > s$ is an interior point.
By \cite{Bhatia23}, Lemma 21 (part 2), there exists $(x_n,t_n) \in \mathbb{Q}^2$ such that $t_n < t$ and $p = (\gamma_n(t),t)$.
Therefore, the interior points of all $\theta$-directed geodesics are determined from the random variable $\mathrm{GT}_{\theta}$ above.
We say a point $q \in \mathrm{GT}_{\theta}$ if $q \in \mathbb{R}^2$ is an interior point of some $\theta$-directed geodesic.

\begin{theorem} \label{thm:main}
   Consider a sequence of directions $\theta^+_n \to + \infty$ and $\theta^-_n \to -\infty$ as $n \to \infty$. Suppose for every $n$, we sample the geodesic tree of the directed landscape for both directions $ \theta^{\pm}_n$. We can reconstruct the directed landscape from all these samples in the sense that $\mathcal{L} $ is a measurable function of the sampled trees.
   
   In other words, there is a function $F : \Gamma^{\mathbb{N}} \to C(\mathbb{R}^4_{\uparrow}, \mathbb{R})$ which is measurable (with respect to the product $\sigma$-algebra induced by $\mathcal{B}_{\Gamma}$ and the Borel $\sigma$-algebra on the target) such that 
   \begin{equation} \label{eqn:F} F(\mathrm{GT}_{\theta_1^+}(s), \mathrm{GT}_{\theta_1^-}(s), \mathrm{GT}_{\theta_2^+}(s), \mathrm{GT}_{\theta_2^-}(s), \cdots) = s\end{equation}
   for $\mu$-almost every $s \in S$ (see \eqref{eqn:dlrv} and \eqref{eqn:geotreedef}).
\end{theorem}

\begin{remark}
    The geodesic tree has been studied in the papers \cite{Bhatia23, BSS, RV}. One caveat is that identifying  $\mathrm{GT}_{\theta}$ as a random subset of the plane consisting of all interior points raises some delicate measurability issues, as one needs to define a $\sigma$-algebra on the space of planar subsets. Another way to define the geodesic tree is to identify it as a random variable on the space of subsets of paths, with the path space suitably compactified. This is how the Brownian web is defined (see \cite{FINR}).
\end{remark} 

\section{The restricted landscape distance}

Let $\theta_1 < \theta_2$ be two directions. The basic idea in what follows is to capture a notion of a coordinate system using geodesics in directions $\theta_1$ and $\theta_2$, and to use said coordinates to decompose geodesics.

\begin{definition} A path  $\pi : [s,t] \to \mathbb{R}$  is called $\theta_1-\theta_2$ wedged if for every $r\in [s,t]$ and $i=1,2$, there are semi-infinite geodesics $\gamma_{i,r}$ in direction $\theta_i$ from $(\pi(r),r)$ so that 
$\gamma_{1,r}(\tau)\le \pi(\tau)\le \gamma_{2,r}(\tau)$ for all $\tau\in [r,t]$.
\end{definition}

\begin{definition}
A $\theta_1-\theta_2$ path is a path $\pi : [s,t] \to \mathbb{R}$ with the following property. The path $\pi$ is a finite union of segments where each segment is part of either a geodesic with direction $\theta_1$ or a geodesic with direction $\theta_2$.
More formally, there exists an $n \geq 1$ and times $s = t_0 < t_1 < \ldots < t_n = t$ such that $\pi$ restricted to $(t_i,t_{i+1}]$ is on a geodesic in direction $\theta_1$ or $\theta_2$, alternating with the parity of $i$.
\end{definition}

The length of a $\theta_1-\theta_2$ path with respect to the directed landscape is the sum of the lengths of its segments (this is easily deduced from the definition of length).

Now we define a restricted distance with respect to the directional parameters $\theta_1$ and $\theta_2$. This distance captures the length of paths as they can be decomposed in the above ``coordinates". 

Let $p = (x,s)$ and $q = (y,t)$ be two points with $s < t$. Set
\begin{equation} \label{eqn:Ltheta12}
    \mathcal{L}^{\theta_1, \theta_2}(p;q) = \sup_{\pi} \; ||\pi||_{\mathcal L}
\end{equation}
where the supremum is over all $\theta_1-\theta_2$ paths $\pi$ from $p$ to $q$.
The length, $||\pi||_{\mathcal L}$, of $\pi$ is with respect to the directed landscape $\mathcal{L}$.
If there are no such paths, set $\mathcal{L}^{\theta_1, \theta_2}(p;q) = -\infty$.
It is clear that $\mathcal{L}^{\theta_1, \theta_2}(p;q) \leq \mathcal{L}(p;q)$.
If there exists a geodesic from $p$ to $q$ that is a $\theta_1-\theta_2$ path, then equality holds.

We may define the length of a path with respect to the distance $\mathcal L^{\theta_1, \theta_2}$ analogous to the length for $\mathcal L$. If $\pi: [s,t] \to \mathbb{R}$ is a path then
$$ ||\pi||_{\mathcal{L}^{\theta_1, \theta_2}} = \inf_{\text{\it all\;partitions}}\; \sum_{i=1}^n L^{\theta_1, \theta_2}\big(\pi(t_i),t_i; \pi(t_{i-1}), t_{i-1}\big)$$
where the infimum is over all partitions $s = t_0 < t_1 < \cdots < t_n = t$ of $[s,t]$.
The path $\pi$ is a $\mathcal L^{\theta_1, \theta_2}$-geodesic if for all $t_1 < t_2 < t_3$ in $[s,t]$, it holds
that $\mathcal L^{\theta_1, \theta_2}(z_1;z_3) = \mathcal L^{\theta_1, \theta_2}(z_1;z_2) + \mathcal L^{\theta_1, \theta_2}(z_2;z_3)$ where $z_i = (\pi(t_i),t_i)$. In this case, $|| \pi ||_{\mathcal L^{\theta_1, \theta_2}} = \mathcal L^{\theta_1, \theta_2}(\pi(s),s; \pi(t),t)$.

As we mentioned above, the distance $\mathcal L^{\theta_1, \theta_2}$ is used to capture the notion of distances between points using a ``coordinate system" of $\theta_1$ or $\theta_2$-directed geodesics. A key observation is that for geodesics in any intermediate direction between $\theta_1$ and $\theta_2$, this system does capture $\mathcal L$-distances faithfully. This is the content of upcoming Proposition \ref{prop:alternatingpath}.

\begin{definition}
    A point $p$ is a forward $k$-star if there are $k$ geodesics emanating from $p$ that are disjoint except for at $p$.
    In other words, there is an $\epsilon > 0$ and geodesics $\gamma_1, \ldots, \gamma_k : [s,s+\epsilon] \to \mathbb{R}$ such that $\gamma_i(t) \neq \gamma_j(t)$
    for every $i,j$ and $t \in (s,s+\epsilon]$, while $(\gamma_i(s),s) = p$ for every $i$.
    
    Similarly, $q$ is a backward $k$-star if there are $k$ geodesics emanating into $q$ that are disjoint except for at $q$.
\end{definition}

The following lemma gathers some facts about $\mathcal{L}$-geodesics from \cite{Dau23}.
\begin{lemma} \label{lem:Dau}
   The following properties of $\mathcal{L}$ hold almost surely.
   \begin{description}
       \item[(i)] Let $\pi_i : [s_i,t_i] \to \mathbb{R}$ for $i=1,2$ be geodesics of $\mathcal{L}$. Define the overlap $O(\pi_1,\pi_2) = \{u \in \mathbb{R}: \pi_1(u) = \pi_2(u)\}$.
       If $O(\pi_1,\pi_2) \neq \emptyset$, there exists $s \leq t$ such that $O = [s,t] \cup S$ where $|S| \leq 2$ and $S \subset \{s_1,s_2,t_1,t_2\}$.
       \item[(ii)] Let $\pi : [s,t] \to \mathbb{R}$ be a geodesic of $\mathcal{L}$. For every $r \in (s,t)$, $(\pi(r),r)$ is not a forward 3-star.
       \item[(iii)] Let $\pi_n :[s_n,t_n] \to \mathbb{R}$ and $\pi : [s,t] \to \mathbb{R}$ be geodesics of $\mathcal{L}$. Let $\Gamma_n = \{(\pi_n(u),u): u \in [s_n,t_n]\}$
       and $\Gamma = \{(\pi(u),u): u \in [s,t]\}$. Suppose $s_n,t_n$ are bounded over $n$. Then $\Gamma_n \to \Gamma$ in the Hausdorff metric
       if and only if $\mathrm{length}(O(\pi_n,\pi)) \to t-s$.
   \end{description}
\end{lemma}

\begin{proof}
    Part (i) is \cite[Lemma 3.3]{Dau23}. Part(ii) is \cite[Lemma 5.5]{Dau23}. Part (iii) is \cite[Proposition 3.5]{Dau23}.
\end{proof}

\begin{proposition} \label{prop:alternatingpath}
    Let $\theta_1< \theta_2$. Let $\gamma:[s,t]\mapsto \mathbb R$ be a geodesic so that 
    \begin{enumerate}
        \item $\gamma$ is $\theta_1-\theta_2$-wedged, 
        \item $p =(\gamma(s),s)$ is not a forward 3-star,
    \item  $q=(\gamma(t),t)$ is not a backward 2-star and 
    \item $q \in \mathrm{GT}_{\theta_1} \cup \mathrm{GT}_{\theta_2}$.
    \end{enumerate}
Then $\gamma$ is a $\theta_1-\theta_2$ path. In particular,
    $$\mathcal{L}^{\theta_1, \theta_2}(p;q) = \mathcal{L}(p;q).$$
\end{proposition}
\begin{remark}
    When $\gamma$ is a finite segment of a geodesic in a direction $\theta \in [\theta_1,\theta_2]$, condition 1 automatically applies by geodesic ordering. 
\end{remark}

\begin{proof}
We will construct the times $t_1,t_2, \ldots, t_n $ as in the definition of a $\theta_1-\theta_2$ path by the greedy algorithm. 

To find $t_1$, consider the leftmost path $\lambda_0$ in $\mathrm{GT}_{\theta_1}$ from $p$ and the rightmost path $\rho_0$ in $\mathrm{GT}_{\theta_2}$ from $p$. Such paths exists by \cite[Corollary 3.20]{RV}.

Since $p$ is not a $3$-star, one of these path must overlap with $\gamma$ until some maximal time $t_1>s$.
That is, $\gamma \mid_{[s,t_1]} = \lambda_0 \mid_{[s,t_1]}$ or $\gamma \mid_{[s,t_1]} = \rho_0 \mid_{[s,t_1]}$.

This is because of the following. First of all, since $\gamma$ is $\theta_1-\theta_2$ wedged,  $\lambda_0 \le \gamma \le \rho_0$. The three paths are not disjoint on any interval $(s,s+\epsilon)$ since $p$ is not a forward $3$-star. Therefore $\gamma$ has an infinite set of intersections with either $\lambda_0$ or $\rho_0$ in $(s,s+\epsilon)$ for every $\epsilon >0$.  Assume w.l.o.g. that it is with $\lambda_0$. Let $O = \{u \geq s: \lambda_0(u) = \gamma(u) \}$.
By Lemma \ref{lem:Dau}(i), $O$ has the form $O = [s,s']\cup S$ where $|S| \leq 2$ and $s' \geq s$.
Thus, $\gamma$ and $\lambda_0$ must agree on some interval $[s,t_1]$ with maximal $t_1>s$, that is, $s' = t_1 > s$. 
Moreover, either $t_1=t$, or $\gamma$ and $\lambda_0$ are disjoint after time $t_1$. 

To find $t_2$, consider the leftmost path $\lambda_1$ in $\mathrm{GT}_{\theta_1}$ from $(\gamma(t_1),t_1)$ and the rightmost path $\rho_1$ in $\mathrm{GT}_{\theta_2}$ from $(\gamma(t_1),t_1)$. Note that $\lambda_1$ is just a restriction of $\lambda_0$ to $[t_1,\infty)$.
 
No internal point of a geodesic is a $3$-star by Lemma \ref{lem:Dau}(ii). So $\gamma$ and $\rho_1$ must have infinitely many intersections on every interval $(t_1,t_1+\epsilon)$. By the same argument as before, there is $t_2>t_1$ such that $\gamma=\rho_1$ on $[t_1,t_2]$. 

Alternating this construction, we either reach the time $t$ in a finite number of steps, or we get an infinite increasing sequence of times $t_i$ converging to some limit $t_\infty\le t$. We will show the latter does not happen. 

Let $\lambda_i$,  $\rho_i$ denote the infinite geodesics arising in this construction.
By compactness \cite[Lemma 14]{Bhatia23}, they admit subsequential limits $\lambda_{\infty}$ and $\rho_{\infty}$ (in the Hausdorff metric for their graphs). These are geodesics defined over $[t_{\infty}, \infty)$ in directions $\theta_1$ and $\theta_2$, respectively. The geodesic $\gamma$ is wedged between them.

If $t_\infty < t$ then we have a contradiction by considering the limiting geodesics $\lambda_{\infty}$ and $\rho_{\infty}$.
Consider the overlaps $O(\gamma, \lambda_{\infty} \mid_{[t_{\infty},t]})$ and $O(\gamma, \rho_{\infty} \mid_{[t_{\infty},t]})$ as in Lemma \ref{lem:Dau}.
If the overlaps equal $\{t_{\infty}\}$ then we find an interior forward $3$-star, which is a contradiction to Lemma \ref{lem:Dau}(ii).
Suppose $\gamma$ overlaps $\lambda_{\infty}: [t_{\infty}, t'] \subseteq O(\gamma, \lambda_{\infty} \mid_{[t_{\infty},t]})$ for some $t' > t_{\infty}$ by Lemma \ref{lem:Dau}(i).
Since geodesic convergence implies overlap convergence, Lemma \ref{lem:Dau}(iii), some $\lambda_i$ overlaps with $\gamma$ after time $t_{\infty}$, which is a contradiction to the way we chose the times to be maximal.

If $t_\infty = t$, we use the fact that $q \in \mathrm{GT}_{\theta_1}\cup \mathrm{GT}_{\theta_2}$ to get a contradiction. W.l.o.g. assume that it is in the first tree. Then it is an internal point on a geodesic $\lambda'$ from some $q'$ in direction $\theta_1$. Since $q$ is not a backward $2$-star, the maximal time interval $[t^-,t]$ on which $\gamma=\lambda'$ satisfies $t^-<t$. In particular, it contains an interval $[t_i,t_{i+1}]$ in which $\gamma$ agrees with the leftmost geodesic $\lambda_i$ in direction $\theta_1$ from $(\gamma(t_i),t_i)$. Now $\lambda_i$ and $\lambda'$ are both geodesics in direction $\theta_1$, and they agree on a time interval $[t_i,t_{i+1}]$ of positive length, so they must agree on $[t_i,\infty)$. Hence by our construction, we would have picked $t_{i+1}=t$, a contradiction. 
\end{proof}

\begin{proposition} \label{prop:geotreeunion}
   Let $\theta_1 \leq \theta \leq \theta_2$, and let $\gamma$ be a geodesic in direction $\theta$ from a point $p=(x,s)$ that is not a forward $3$-star. Then for any $q=(\gamma(t),t)$ with $t>s$, the $pq$-geodesic $\gamma|_{[s,t]}$ is a $\theta_1-\theta_2$-path, and $$\|\gamma|_{[s,t]}\|_\mathcal{L}=\mathcal{L}(p;q) = \mathcal{L}^{\theta_1, \theta_2}(p;q).$$
\end{proposition}

\begin{proof}
    
    We may find a point $q' = (\gamma(t'),t')$ for $t' > t$ such that (i) $q' \in \mathrm{GT}_{\theta_1} \cup \mathrm{GT}_{\theta_2}$, and (ii) $q'$ is not a backward 2-star. Indeed, let $S$ be the set of star points in $\gamma$, which has Hausdorff dimension 1/3 \cite[Theorem 4]{Bhatia22}. Choose a point $q''$ in $\gamma$ with time coordinate larger than $t$ and which does not belong to $S$. Consider a geodesic $\gamma''$ in direction $\theta_1$ from $q''$. The intersection of $\gamma''$ with $\gamma$ must be an interval of positive length because $q''$ is not a star. Choose $q'$ to be any point from this interval that is not in $S$.

    Since $p$ is not a forward 3-star, by Proposition \ref{prop:alternatingpath}, $\gamma|_{[s,t']}$ is a $\theta_1-\theta_2$ path, and so is its restriction $\gamma|_{[s,t]}$. Thus, by definition, 
    $$\mathcal{L}(p;q) =\|\gamma|_{[s,t]}\|_{\mathcal L}= \mathcal{L}^{\theta_1, \theta_2}(p;q).$$
\end{proof}

\begin{corollary} \label{cor:theta}
Let $(\theta_1', \theta_2')\subseteq (\theta_1,\theta_2)$, and assume that $p$ is not a forward $3$-star. Then any  $\theta_1'-\theta_2'$ path from $p$
 is also a $\theta_1-\theta_2$ path. 
\end{corollary}

\begin{proof}
Recall that such paths $\pi$ are alternating between geodesic segments in the direction $\theta_1'$ and $\theta_2'$. Such segments are themselves $\theta_1-\theta_2$ path as long as their initial point $p_i$ is not a forward $3$-star. This is true for the first segment by assumption, while the other $p_i$ are interior points that are never forward 3-stars by \cite{Dau23}. 
\end{proof}
 
   For the next part of the argument, we will need to consider competition interfaces.
   For $q=(y,t)\in \mathbb R^2$ and $\theta\in \mathbb R$, let $I_{q,\theta}:[-\infty,t]\to\mathbb R$ be a competition interface from $q$ in direction $\theta$.
   See \cite[Lemma 31]{Bhatia23} for the construction of this interface (see also \cite[Lemma 4.20]{RV}). Roughly speaking, $I_{q,\theta}$ is a path
   from $q$ moving downwards with direction $\theta$ such that it does not intersect $\mathrm{GT}_{\theta}$ after $q$.
   
   \begin{lemma} \label{lem:interface}
   Almost surely, for every $q=(y,t)$ and $r<t$, we have 
   $$
   \lim_{\theta\to \infty} \;\;\sup_{(-\infty ,r]}I_{q,\theta}=-\infty. 
   $$
   \end{lemma}
   \begin{proof}
    First we prove this for fixed $q$. By translation-invariance of $\mathcal{L}$ we may assume $q=(0,0)$.
      Since $I_{(0,0),\theta}$ is a monotone function of $\theta$ ($I_{(0,0),\theta_1} \geq I_{(0,0),\theta_2}$ when $\theta_1 \leq \theta_2$),
      it suffices to show that along the sequence $\theta = n \to +\infty$,
      $$
      \sup_{(-\infty ,r]}I_{(0,0),n} \to -\infty \quad \text{in probability}.
      $$
      Indeed, convergence in probability will imply convergence almost surely due to the monotonicity.
      The latter convergence depends on the individual law of $I_{(0,0),n}$ for each $n$, and not on how they are coupled over $n$.
      By shear-invariance of the directed landscape,
      $$
      (I_{(0,0), n}(s),s<0)\stackrel{d}{=} (I_{(0,0),0}(s)+n s,s<0)
      $$ 
      We have, for $r < 0$,
      \begin{align*}
      \sup_{s\in (-\infty ,r]} I_{(0,0),0}(s)+n s
      &\le
      \left(\sup_{s\in (-\infty ,r]} I_{(0,0),0}(s)+s\right)+ \left(
      \sup_{s\in (-\infty ,r]} (n-1) s\right)
\\& =   \left(\sup_{s\in (-\infty ,r]} I_{(0,0),0}(s)+ s\right)
       +(n-1) r,
       \qquad n \ge 1
      \end{align*}
The last infimum is finite by \cite[Corollary 4.14]{RV}, which states that $$\lim_{s\to -\infty} \frac{I_{(0,0),0}(s)}{s}= 0.$$
So the first supremum converges to $-\infty$ as $n \to\infty$ since $r < 0$.  
(In the aforementioned Corollary 4.14 we use that the initial condition for the interface $I_{(0,0),0}$
is the Busemann function $W_0(x,0;0,0)$ \eqref{eqn:Busemann} -- a two-sided Brownian motion that grows sublinearly.)

   By a union bound, the statement of the lemma holds for all $q$ with rational coordinates. The rest follows by the ordering of competition interfaces. Indeed, let $q'=(y',t')$ be rational with $r < t' < t$ and $y' > I_{q,\theta}(t')$. Then $I_{q,\theta} \leq I_{q',\theta}$ on $(-\infty,r]$. 
\end{proof}

\begin{proposition}\label{prop:density}
    Let $\gamma:[s,t]\mapsto \mathbb R$ be a geodesic so that $(\gamma(s),s)$ is not a forward 3-star. Then for every $r\in (s,t)$ there exists $\theta_1,\theta_2$ so that $\gamma|_{[s,r]}$ is a $\theta_1-\theta_2$ path. 

Let $U$ be the set of point pairs $(p;q)$ for which $p$ is not a forward 3-star and there is a $pq$-geodesic that is also a $\theta_1-\theta_2$ path for some $\theta_1,\theta_2$. Then $U$ is dense in $\mathbb R_{\uparrow}^4.$
\end{proposition}

\begin{proof}

  Let $q=(\gamma(t),t)$, and let $r'\in (r,t)$
   By Lemma \ref{lem:interface} (and symmetry) we can find $\theta_1,\theta_2$ so that $I_{q,\theta_2}< \gamma < I_{q,\theta_1}$ on the interval $[s,r']$. This implies that $\gamma|_{[s,r']}$ is $\theta_1-\theta_2$ wedged. Indeed, let $\tau\in [s,r']$ and let $\lambda_\tau$ be a leftmost $\theta_1$-geodesic from $(\gamma(\tau),\tau)$ in the direction $\theta_1$. Then $\lambda_{\tau}(\tau)=\gamma(\tau)<I_{q,\theta_1}(\tau)$, and therefore, since geodesics do not cross competition interfaces, this ordering also holds at time $t$, namely 
   $\lambda_{\tau}(t)\le I_{q,\theta_1}(t)=\gamma(t)$.
   
   We see that $\lambda_{\tau}\le \gamma$ at times $\tau$ and $t$, and since  $\lambda$ is a leftmost geodesic, it follows that $\lambda_\tau\le \gamma$ on the entire interval $[\tau,t]$. After a symmetric argument for $\theta_2$, we see that $\gamma|_{[s,r']}$ is $\theta_1-\theta_2$ wedged. Now we a find $r''\in [r,r']$ so that $(\gamma(r''),r'')$ is not a backward two star. Then by Proposition \ref{prop:alternatingpath}, $\gamma|_{[s,r'']}$ is a $\theta_1-\theta_2$-path, and so is its restriction to $[s,r]$.

   The second claim is immediate. 
\end{proof}

\section{Constructing trees from the restricted landscape distance}

A continuous function $\gamma : [s,\infty) \to \mathbb{R}$ is a $\mathcal{L}^{\theta_1, \theta_2}$ semi-infinite geodesic if for every $t > s$, $\gamma \mid_{[s,t]}$ is a $\mathcal{L}^{\theta_1, \theta_2}$ geodesic. The direction of $\gamma$ is $\lim_{t \to \infty} \gamma(t)/t$, should it exist.

\begin{lemma} \label{lem:uniquemodgeodesic}
Let $\theta_1 \leq \theta \leq \theta_2$.
Suppose there is a unique semi-infinite $\mathcal{L}$-geodesic $\gamma$ in direction $\theta$ starting from a point $p \in \mathbb{R}^2$ that is not a forward $2$-star.
Then $\gamma$ is also the unique semi-infinite $\mathcal{L}^{\theta_1, \theta_2}$-geodesic from $p$ in direction $\theta$.
\end{lemma}

\begin{proof}
Firstly, $\gamma$ is an $\mathcal{L}^{\theta_1, \theta_2}$-geodesic by Proposition \ref{prop:alternatingpath}.

If $\theta = \theta_1, \theta_2$, the claim is trivial. So assume that $\theta_1 < \theta < \theta_2$.
Suppose $\gamma'$ is another $\mathcal{L}^{\theta_1, \theta_2}$-geodesic from $p$ in direction $\theta$.
Let $q$ be any point on the intersection of $\gamma$ and $\gamma'$. Let $\gamma_{pq} $ and $\gamma'_{pq}$ denote the segments of $\gamma$ and $\gamma'$ from $p$ to $q$, respectively. Denote by $||\pi||_{d}$ the length of the path $\pi$ with respect to a directed metric $d$; for this proof, $d$ will be either $\mathcal{L}$ or $\mathcal{L}^{\theta_1, \theta_2}$.

Now, we observe that 
$$
||\gamma'_{pq}||_{\mathcal{L}}
\ge ||\gamma'_{pq}||_{\mathcal{L}^{\theta_1, \theta_2}}
= ||\gamma_{pq}||_{\mathcal{L}^{\theta_1, \theta_2}}
$$
By Proposition \ref{prop:alternatingpath}, this equals 
$$
||\gamma_{pq}||_{\mathcal{L}}=\mathcal L(p;q).
$$
So $\gamma'_{pq}$ is a $\mathcal L$-geodesic. Concatenating $\gamma'_{pq}$ with $\gamma_{q \infty}$, where $\gamma_{q \infty}$ is the segment of $\gamma$ from $q$ onwards, we get an $\mathcal L$-geodesic $\pi$ in the direction $\theta$ from $p$. To see this, let $r$ be later than $q$ on $\pi$. We have
\begin{align*}
||\pi_{pr}||_{\mathcal{L}}
&=||\pi_{pq}||_{\mathcal{L}}+||\pi_{qr}||_{\mathcal{L}}
=||\gamma'_{pq}||_{\mathcal{L}}+||\gamma_{qr}||_{\mathcal{L}}
\\ 
&=||\gamma_{pq}||_{\mathcal{L}}+||\gamma_{qr}||_{\mathcal{L}}
=||\gamma_{pr}||_{\mathcal{L}}=\mathcal L(p;r)
\end{align*}
By uniqueness of geodesics, $\gamma=\pi$, and so $p=q$, that is $\gamma'$ only intersects $\gamma$ in the point $p$. We may hence assume that $\gamma' \le \gamma$.

Let $\theta'\in (\theta_1,\theta)$ be a direction so that there is a unique infinite geodesic $\lambda$ from $p$ in the direction $\theta'$. For large enough $t$, $\lambda(t)<\gamma'(t)$, but since $p$ is not a forward 2-star, for some $t$ slightly greater than the time of $p$,  $\gamma'(t)<\gamma(t)= \lambda(t)$.

By the intermediate value theorem, $\lambda$ must intersect $\gamma'$ at some point $q \neq p$. Since $q$ is on the infinite geodesic $\lambda$, we have $\mathcal L^{\theta_1, \theta_2}(p;q)=\mathcal L(p;q)$ by Proposition \ref{prop:geotreeunion}. Thus
$$
||\gamma'_{pq}||_{\mathcal L^{\theta_1, \theta_2}}=\mathcal L^{\theta_1, \theta_2}(p;q)=\mathcal L(p;q),
$$
so $\gamma'_{pq}$ is an $\mathcal L$-geodesic, showing that $p$ is a forward $2$-star, a contradiction. 
\end{proof}

\begin{theorem} \label{thm:thetareconstruct}
Let $\theta_1 \leq \theta \leq \theta_2$.
We can construct the geodesic tree $\mathrm{GT}_{\theta}$ using the $\mathcal L^{\theta_1,\theta_2}$ distance. 
\end{theorem}

\begin{proof}
 Consider all rational starting points for $\theta$-directed $\mathcal{L}$-geodesics. These will not be $2$-stars, and will be points of uniqueness for $\theta$-directed $\mathcal{L}$-geodesics. The tree $\mathrm{GT}_{\theta}$ is determined from these geodesics by approximation (see \eqref{eqn:geotreedef} and the surrounding discussion). Lemma \ref{lem:uniquemodgeodesic} implies these geodesics may be constructed using the $\mathcal L^{\theta_1,\theta_2}$ distance. 
\end{proof}

\section{The differential distance}
Let $\theta \in \mathbb{R}$ be a direction. Let $W_{\theta}$ denote the Busemann function in direction $\theta$. See \cite{BSS, GZ, RV} for the construction of the Busemann function. It may be defined as
\begin{equation} \label{eqn:Busemann}
    W_{\theta}(p;q) = \lim_{t \to \infty} \mathcal{L}(p; \theta t, t) - \mathcal{L}(q; \theta t, t).
\end{equation}

Define the function
\begin{equation} \label{eqn:D}
\mathcal{D}_{\theta}(p;q) = W_{\theta}(p;q) - \mathcal{L}(p;q).
\end{equation}
Here $p,q \in \mathbb{R}^2$ and we understand that $\mathcal{D}_{\theta}(p;q) = +\infty$ if $p=(x,s)$ and $q = (y,t)$ with $s > t$.

We call the function $\mathcal{D}_{\theta}$ the {\bf differential distance} in direction $\theta$.

Given a path $\pi: [s,t] \to \mathbb{R}$, the length of $\pi$ with respect to $\mathcal{D}_{\theta}$ is
\begin{equation} \label{eqn:differentiallength}
||\pi||_{\mathcal{D}_{\theta}} = \sup_{s=t_0 < t_2 < \cdots < t_n = t} \sum_{i=1}^n \mathcal{D}_{\theta}(\pi(t_i),\pi(t_{i-1}))
\end{equation}
where the supremum is over all partitions of the interval $[s,t]$. A geodesic $\gamma$ of $\mathcal{D}_{\theta}$ from $p$ to $q$ is a path from $p$ to $q$ of \emph{minimal} length.

\begin{lemma} \label{lem:DandLlength}
    For any path $\pi: [s,t] \to \mathbb{R}$, it holds that
    $$||\pi||_{\mathcal{D}_{\theta}} = W_{\theta}(p;q) - ||\pi||_{\mathcal{L}}$$
    where $p = (\pi(s),s)$ and $q = (\pi(t),t)$. In particular, geodesics of $\mathcal{D}_{\theta}$ are the same as the geodesics of $\mathcal{L}$. Furthermore, if $\gamma$ is a geodesic of $\mathcal{D}_{\theta}$ from $p$ to $q$, then
    $$ ||\gamma||_{\mathcal{D}_{\theta}} = W_{\theta}(p;q) - ||\gamma||_{\mathcal{L}} = W_{\theta}(p;q) - \mathcal{L}(p;q) = \mathcal{D}_{\theta}(p;q).$$
\end{lemma}

\begin{proof}
    The claims follows from the fact that $W_{\theta}$ is additive: $W_{\theta}(p;q) = W_{\theta}(p;r) + W_{\theta}(r;q)$.
\end{proof}

We also define, for $\theta, \theta_1, \theta_2 \in \mathbb{R}$ with $\theta_1 < \theta_2$,
\begin{equation} \label{eqn:Drelative}
\mathcal{D}^{\theta_1,\theta_2}_{\theta}(p;q) = W_{\theta}(p;q) - \mathcal{L}^{\theta_1, \theta_2}(p;q).
\end{equation}
Then definition \eqref{eqn:Ltheta12} and Lemma \ref{lem:DandLlength} imply
\begin{equation} \label{eqn:Dtheta12alt}
    \mathcal{D}^{\theta_1, \theta_2}_\theta(p;q) = \inf_{\pi} \; ||\pi||_{\mathcal D_\theta}
\end{equation}
where the infimum is over all $\theta_1-\theta_2$ paths $\pi$ from $p$ to $q$.

\begin{remark}\label{rem:DandLlength12}
By the same argument as in Lemma \ref{lem:DandLlength}, $\mathcal{D}^{\theta_1, \theta_2}_\theta(p;q)$ and $\mathcal{L}^{\theta_1, \theta_2}(p;q)$ have the same geodesics for any $\theta \in \mathbb{R}$. By Lemma \ref{lem:uniquemodgeodesic}, the unique $\mathcal{L}^{\theta_1, \theta_2}$ geodesics from rational points in any fixed direction $\theta \in [\theta_1,\theta_2]$ are exactly the unique $\mathcal{L}$ geodesics from those points in direction $\theta$. In particular, we can construct $\mathrm{GT}_{\theta}$ from $\mathcal{D}^{\theta_1, \theta_2}_{\theta_2}$ and the trees $\mathrm{GT}_{\theta_i}$, $i=1,2$, as in Theorem \ref{thm:thetareconstruct}.
\end{remark}

Recall that the Busemann function evolves as a KPZ Fixed Point in the following sense, see \cite[Theorem 3.23]{RV} and \cite[Theorem 5.1(iv)]{BSS}.
\begin{theorem} \label{thm:busemannkpz}
For $s < t$, one has that
$$W_{\theta}(x,s; y,t) = \sup_{z \in \mathbb{R}} \{ \mathcal{L}(x,s; z,t) + W_{\theta}(z,t; y,t)\}.$$
Furthermore, any maximizer $z$ is such that $(z,t)$ lies on a geodesic from $(x,s)$ in direction $\theta$.
\end{theorem}

\begin{definition} \label{def:ancestral}
Let $\theta \in \mathbb{R}$. A pair of points $(p;q)$ with $p = (x,s)$ and $q = (y,t)$ is said to be ancestral in direction $\theta$ if $s \leq t$ and they lie on a common geodesic in direction $\theta$.
\end{definition}

\begin{lemma} \label{lem:Dtheta}
The function $\mathcal{D}_{\theta}$ satisfies $\mathcal{D}_{\theta}(p;q) \geq 0$ and $\mathcal{D}_{\theta}(p;q) = 0$ if and only if $(p;q)$ is ancestral in direction $\theta$.
The function $\mathcal{D}_{\theta}$ also satisfies the triangle inequality: $\mathcal{D}_{\theta}(p;q) \leq \mathcal{D}_{\theta}(p;r) + \mathcal{D}_{\theta}(r;q)$.
\end{lemma}

\begin{proof}
Using $z=y$ in the characterisation of $W_{\theta}$ in Theorem \ref{thm:busemannkpz} we see that $W_{\theta}(x,s;y,t) \geq \mathcal{L}(x,s;y,t)$
when $s < t$. If $s \geq t$ and $x \neq y$ then we have $\mathcal{D}_{\theta}(p;q) = +\infty$.
Finally, $\mathcal{D}_{\theta}(p;p) = W_{\theta}(p;p) - \mathcal{L}(p;p) = 0$.

Suppose $\mathcal{D}_{\theta}(p;q) = 0$ with $p = (x,s)$ and $q = (y,t)$. If $p=q$ then they are clearly ancestral. Otherwise, $s < t$.
We have that
$$W_{\theta}(x,s; y,t) = \sup_{z \in \mathbb{R}} \{ \mathcal{L}(x,s; z,t) + W_{\theta}(z,t; y,t)\} = \mathcal{L}(x,s;y,t).$$
Thus, $z= y$ is a maximizer above. From the ``furthermore" part of Theorem \ref{thm:busemannkpz} we find that $q=(y,t)$ lies
on a geodesic from $p=(x,s)$ in direction $\theta$. So $(p;q)$ is ancestral.

Since $W_{\theta}$ is additive and $\mathcal{L}$ satisfies the reverse triangle inequality, $\mathcal{D}_{\theta}$ satisfies the triangle inequality.
\end{proof}

By Lemma \ref{lem:Dtheta}, $||\pi||_{\mathcal{D}_{\theta}} = 0$ if and only if $\pi$ is a segment from the geodesic tree in direction $\theta$.

\begin{lemma} \label{lem:DandGT}
The directed landscape $\mathcal{L}$ is a measurable function of $\mathcal{D}_{\theta}$.
\end{lemma}

\begin{proof}
By continuity, $\mathcal L$ is determined by its values at rational time pairs $s<t$. For fixed $s<t$, the law of $\mathcal A(\cdot,\cdot)=\mathcal L(\cdot,s;\cdot,t)$ is called the Airy sheet at scale $(t-s)^{1/3}$. The Airy sheet is the cumulative distribution function of the {\bf shock measure} $\mu_{s,t}$, see \cite{DV} Remark 4.5:
\begin{align*}
\mu_{s,t}([x_1,x_2] \times [y_1,y_2])&=\mathcal A(x_1,y_1)+\mathcal A(x_2,y_2)-\mathcal A(x_1,y_2)-\mathcal A(x_2,y_1)\\&=\mathcal A_\theta(x_1,y_2)+\mathcal A_\theta(x_2,y_1) -\mathcal A_\theta(x_1,y_1)-\mathcal A_\theta(x_2,y_2),
\end{align*}
where $\mathcal{A}_\theta(\cdot,\cdot)=\mathcal D_\theta(\cdot,s;\cdot,t)$, and the second equality holds since the Busemann function is additive. Thus $\mathcal{D}_\theta$ determines $\mu_{s,t}$. Also $\mu_{s,t}$ determines $\mathcal A$ by ergodicity; the argument is part of the proof of Theorem 1.5 in \cite{Dau21}.
\end{proof}

For the next lemma, let $\theta_1 < \theta_2$ be two directions. Define the function
\begin{equation} \label{eqn:delta}
\Delta(x,t) = W_{\theta_2}(x,t;0,0) - W_{\theta_1}(x,t;0,0).
\end{equation}
\begin{lemma} \label{lem:relativeD}
    Consider two directions $\theta_1 < \theta_2$.
    \begin{enumerate}
    \item
    For ancestral points $(p;q)$ on a $\theta_2$-directed geodesic, we have that $\mathcal{D}_{\theta_2}(p;q) = 0$.
    \item
    For ancestral points $(p;q)$ on a $\theta_1$-directed geodesic, we have that $\mathcal{D}_{\theta_2}(p;q) = \Delta(p) - \Delta(q)$.
    \item 
    The tree structures $\mathrm{GT}_{\theta_1},$ $\mathrm{GT}_{\theta_2}$ and the function $\Delta$ determine the distance $\mathcal{D}^{\theta_1,\theta_2}_{\theta_2}$.
    \end{enumerate}
\end{lemma}

\begin{proof}
Part 1: it is one of the claims of Lemma \ref{lem:Dtheta}. Part 2: note that $D_{\theta_2}(p;q)=D_{\theta_1}(p;q)+W_{\theta_2}(p;q)-W_{\theta_1}(p;q)$, and for ancestral points in direction $\theta_1$ the first term vanishes by Part 1. Then $W_{\theta_2}(p;q)-W_{\theta_1}(p;q)=\Delta(p)-\Delta(q)$ by the additivity of Busemann functions, a direct consequence of the definition.  Part 3: it follows since the $D^{\theta_1,\theta_2}_{\theta_2}$ distance is defined by the $D_{\theta_2}$-lengths of $\theta_1-\theta_2$ paths as in \eqref{eqn:Dtheta12alt}. The $D_{\theta_2}$-length of such a path, by definition,  is  just  the sum of  $D_{\theta_2}$-lengths of the geodesic segments in the directions $\theta_1$ and $\theta_2$. These are determined by 1 and 2. 
\end{proof}

\section{The Busemann difference profile}
In this section, all Brownian motions are two-sided and have variance 2 unless stated otherwise.

The main result of this section is that we can determine the function $\Delta(x,t)$ in \eqref{eqn:delta}  as a measurable function of the geodesic trees $\mathrm{GT}_{\theta_1}$ and
$\mathrm{GT}_{\theta_2}$.

Firstly, we recall the joint law of the pair $\big(W_{\theta_1}(x,t;0,0), W_{\theta_2}(x,t;0,0)\big)$. This is given by the stationary horizon; see \cite[Theorem 5.3(iii)]{BSS} (see also \cite[Section 4.2]{Busani} and  \cite[Section 9.2]{SS}. We record it in the following lemma.

\begin{lemma} \label{lem:coupling}
   Consider the functions $x \mapsto W_{\theta_i}(x,t;0,0) - W_{\theta_i}(0,t;0,0) = W_{\theta_i}(x,t;0,t)$. Let $B_1$ be a Brownian motion with drift $2\theta_1$.
   Let $B_3$ be an independent Brownian motion with drift $2\theta_2$. 
    Let $B_2$ be the process $B_1$ reflected off of the process $B_3$, defined as
   $$B_2(x) = B_1(x) + \sup_{y \leq x} \{B_3(y)-B_1(y)\} - \sup_{y \leq 0} \{B_3(y)-B_1(y)\}.$$
   Then for any $t$ fixed, $(B_1,B_2)$ has the law of $\big (W_{\theta_1}(x,t;0,t), W_{\theta_2}(x,t;0,t)\big )$.
   $B_2$ is a Brownian motion with drift $2\theta_2$. In particular, for every fixed $t$, as a process in $x$,
   \begin{equation}\label{Delta} \Delta(x,t) +W_{\theta_2}(0,0;0,t) - W_{\theta_1}(0,0;0,t) \stackrel{d}{=}
   \sup_{y \leq x} \{B_3(y)-B_1(y)\} - \sup_{y \leq 0} \{B_3(y)-B_1(y)\} .\end{equation}
\end{lemma}
Equation \eqref{Delta} shows that for any fixed $t$, $x\mapsto \Delta(x,t)$ is a random constant plus the running maximum of a Brownian motion with drift $2(\theta_2-\theta_1)$, and so we have the following. 
 
\begin{lemma} \label{lem:W-V}
For fixed $t \in \mathbb{R}$, the function $x \mapsto \Delta(x,t)$ is non-decreasing and $\lim_{x \to \pm \infty} \Delta(x,t) = \pm \infty$.
\end{lemma}

By Lemma \ref{lem:W-V} and continuity of $\Delta$ we get
\begin{corollary} \label{cor:levelset}
Fix $t \in \mathbb{R}$. For every $a \in \mathbb{R}$ there are finite quantities $L^{-}_a(t) \leq L^{+}_a(t)$ such that
$$\{x \in \mathbb{R}: \Delta(x,t) = a\} = [ L^{-}_a(t), L^{+}_a(t)].$$
\end{corollary}

\begin{lemma} \label{lem:ordering}
    If $a < b$ then $L_a^{+}(t) < L_b^{-}(t)$ for every $t$.
\end{lemma}
\begin{proof}
Let $x = L_a^{+}(t)$ and $y = L_b^{-}(t)$. We have $\Delta(x,t) = a < b = \Delta(y,t)$.
Since $\Delta$ is monotone in the first argument, it follows that $x < y$.
\end{proof}

\begin{definition} \label{def:deltasim}
   Fix $\theta_1<\theta_2$. Define a relation on $\mathbb{R}^2$ as follows. We say $p_1 \sim p_2$ if there is a point $r$ such that any geodesic in direction $\theta_i$ from $p_j$ passes through $r$  for $i=1,2$ and $j=1,2$.
   \end{definition}

If $p \sim q$ then $\Delta(p) = \Delta(q)$ because, then, $W_{\theta_2}(p;q) = W_{\theta_1}(p;q)$. Conversely, if $p = (x,s)$ and $q = (y,s)$ and $\Delta(p) = \Delta(q)$, then $p \sim q$ by Theorem 7.8 in \cite{BSS}. In particular, for every time $s \in \mathbb{R}$, the partition of $\mathbb{R}$ formed by the intervals $[L_a^-(s), L_a^+(s)]$ for $a \in \mathbb{R}$ is determined from the geodesic trees $\mathrm{GT}_{\theta_1}$ and $\mathrm{GT}_{\theta_2}$. In other words, the support of the measure associated to the monotone function $\Delta$ is determined by the two trees.

\begin{lemma} \label{lem:deltareconstruct}
    We can determine the function $\Delta$ from the geodesic trees $\mathrm{GT}_{\theta_1}$ and $\mathrm{GT}_{\theta_2}$.
\end{lemma}

\begin{proof}
For a fixed time $s$, the closed intervals $[L_a^-(s), L_a^+(s)]$, $a \in \mathbb{R}$, form a partition of $\mathbb{R}$ that is determined from $\mathrm{GT}_{\theta_1}$ and $\mathrm{GT}_{\theta_2}$ by Theorem 7.8 in \cite{BSS} (as mentioned above). 
These are exactly the intervals of constancy (or plateaus) of the Busemann difference profile $\Delta$.
Recall from Lemma \ref{lem:coupling} that $\Delta(x,s)= \sup_{y\le x} \{B_4(2y)+2(\theta_2-\theta_1)y\} + C_s$ where $B_4$ is a standard Brownian motion and $C_s$ is a random constant.
Then by Lemma \ref{lem:taylorwendel} below, we can reconstruct $x \mapsto \Delta(x,s)-\Delta(0,s)$ from the plateaus.

The above determines $\Delta(\cdot, s) - \Delta(0,s)$ at all rational times $s$. Since $\Delta$ is continuous, we can extend it to all times. Next, we show how to remove the $\Delta(0,s)$ term and determine $\Delta$ globally.
Note that $\Delta$ is constant on the equivalence classes defined in Definition \ref{def:deltasim} (this equivalence relation is finer than the level set one but can also be determined from the trees $\mathrm{GT}_{\theta_1}$ and $\mathrm{GT}_{\theta_2}$). So it suffices to show that for all rational times $s \in \mathbb{R}$, there are points $x,y \in \mathbb{R}$ so that $(x,0)\sim (x+y,s)$. Then we will have $\Delta(x,0) = \Delta(x+y,s)$ and, therefore, $(\Delta(x,0) - \Delta(0,0)) - (\Delta(x+y,s)-\Delta(0,s)) = \Delta(0,s) - \Delta(0,0) = \Delta(0,s)$. Thus, by continuity, we can determine $\Delta(0,s)$ for every $s$ and recover $\Delta$.

We may assume w.l.o.g. that $s > 0$. By ergodicity of the directed landscape under spatial shifts, see \cite[Lemma B.4]{BSS}, it suffices to show that for every $s >0$ there exists $y$ so that $(0,0)\sim (y,s)$ with positive probability.
Indeed, if $(0,0) \sim (y,s)$ with positive probability then ergodicity implies there almost surely exists an $x \in \mathbb{R}$ such that $(x,0) \sim (x+y,s)$.

Let $y>0$. The event $(0,0) \sim (y,s)$ happens if the following three events occur. The event $A_y$ is that $(0,0)\sim (2y,0)$. The event $E_{y,s}$ is  that the geodesic in direction $\theta_1$ from $(0,0)$ is at most $y$ at time $s$. 
The event $E_{y,s}'$ is  that the geodesics in direction $\theta_1$ from $(2y,0)$ is at least $y$ at time $s$. Call the first intersection point of these two geodesics $r$.

On $A_y$, the geodesics in direction $\theta_2$ from $(0,0)$ and $(2y,0)$ both pass through $r$, and until the time of $r$ they must agree with the corresponding geodesics in direction $\theta_1$. 

On $E_{y,s}\cap E'_{y,s}$, by geodesic ordering, the geodesic from $(y,s)$ in directions $\theta_1$ and $\theta_2$ must also pass through $r$. This now implies that $(y,s) \sim (0,0)$.

There are constants $c_2$ and $b_2$, depending of $\theta_1$, such that 
$$P(E_{y,s}^c) \le c_2e^{-b_2(y^3/s^2)}, \quad P({E'}_{y,s}^c) \le c_2e^{-b_2(y^3/s^2)}.$$
For $s=1$ this follows from the tail bound on infinite geodesics given in \cite[Theorem 3.12]{RV}. For general $s$, the bound follows from KPZ scaling. 

By  \cite[Theorem 7.8]{BSS}, $A_y$ holds as long as $\Delta(2y,0) = \Delta(0,0) = 0$, that is, $W_{\theta_1}(2y,0;0,0) = W_{\theta_2}(2y,0;0,0)$.
Its probability is calculated in \cite[Theorem 3.11]{SS} and \cite[Theorem 2.4.1]{Sorensen} using the notion of the stationary horizon.
We quote the result. Let $\xi = \theta_2-\theta_1$ and $\Phi$ be the CDF of the standard Normal distribution.
$$ P(A_y) = \Phi(- \xi \sqrt{2y}) + (1 + \xi^2 4y)\Phi(- \xi \sqrt{2y}) - \xi \frac{2\sqrt{y}}{\sqrt{\pi}} e^{- \xi^2 y}.$$
It follows from the expression above that there are positive constants $c_1, b_1$ that depend on $\theta_i$ such that
$$ P(A_y) \geq c_1 e^{-b_1 y}.$$

Putting these together, and using a union bound, we have
$$
P(A_{y}\cap E_{y,t}\cap E_{y,t}')\ge  c_1e^{-b_1y} - 2c_2e^{-b_2(y^3/t^2)}
$$
and the right hand side is positive as long as $y$ is large enough, as required.
\end{proof}

\begin{lemma}  \label{lem:taylorwendel}
Let $B$ be two-sided Brownian motion with drift $a>0$, and  
$S(x)=\sup_{-\infty\le y\le x} B(y)$. Set
$$K=\mathbb R\setminus \bigcup_{x<y,S(x)=S(y)} (x,y).$$
Set $K$ is the support of the measure associated to the monotone function $S$.
Almost surely, for all $x<y$,
$$S(y)-S(x)=m_\varphi([x,y]\cap K),$$
where $m_\varphi$ is the Hausdorff measure for the gauge function $\varphi(h)=(h\log |\!\log h|)^{1/2}$.
\end{lemma}

\begin{proof}
This is an application of Theorem 1 of \cite{taylor1966exact}.
See also \cite{perkins1981exact, greenwood1980construction} about how to reconstruct Brownian local time from its support.

Note that $S(x) < \infty$ due to $B$ having positive drift.
By translation-invariance and continuity, it suffices to show the result for all sufficiently large rational $x$ and rational $y>x$. 

Set $M(x)=\max_{0\le y \le x}B(y)$.
Note that for $x\ge 0$,  $S(x)=\max(S(0),M(x)).$ Let $X \in [0,\infty)$ be the first non-negative time such that $B(X)\ge S(0)$. This is finite since $B$ has positive drift. Then for all $x\ge X$, $S(x)=M(x)$.  So for $y > x \geq X$,
$$ K \cap [x,y] = \{t \in [x,y]: t\;\text{is a point of increase of}\; M(\cdot)\}.$$
But the points of increase of $M$ are precisely the set of record times, that is,
$K\cap[x,y]$ equals $\{t\in[x,y]:B(t)=M(t)\}$. 

So it suffices to show that for any rational $0<x<y$ we have, almost surely, 
\begin{equation}\label{e:mB}
m_\varphi\{t\in [x,y]:B(t)=M(t)\}=M(y)-M(x)\qquad \text{a.s}
\end{equation}

Since \eqref{e:mB}  is an almost sure statement, and the law of drifted Brownian motion on $[0,y]$ is absolutely continuous with respect to Brownian motion with no drift, it suffices to show the claim \eqref{e:mB} when $B$ is Brownian motion without drift. So for the rest of the proof, let $B, M$ be associated to Brownian motion with zero drift ($a=0$). 

In the driftless scenario, by Levy's theorem (Theorem 6.10 \cite{morters2010brownian}),
 we have the equality in law:
 \begin{align}\notag
\Big((M(t),&t\in[x,y]),\{t\in[x,y]:B(t)=M(t)\}\Big)\;\\ &\stackrel{d}{=}\;\Big(L(t),t\in[x,y]),\{t\in [x,y]:B(t)=0\}\Big).\label{e:mb2}
\end{align}
where $L(t)$ is local time process of $B$ at zero. 
By Theorem 1 of \cite{taylor1966exact}, see Theorem 6.41 in \cite{morters2010brownian},
we have \begin{equation}\label{e:mb3}
    m_\varphi\Big(\{t\in [x,y]:B(t)=0\}\Big)=L(y)-L(x) \qquad \text{a.s.,}
\end{equation} and \eqref{e:mb2} and \eqref{e:mb3} imply \eqref{e:mB}, as required.
\end{proof}

\section{Proof of Theorem \ref{thm:main}}
\begin{proof}
Let $\theta_1 \leq \theta \leq \theta_2$. Suppose we are given samples of the geodesic trees $\mathrm{GT}_{\theta_1}$ and $\mathrm{GT}_{\theta_2}$.

Recall the differential distance $\mathcal{D}_{\theta}$ from \eqref{eqn:D} and the restricted differential distance $\mathcal{D}_{\theta}^{\theta_1,\theta_2}$ from \eqref{eqn:Drelative}.

From these samples we can reconstruct the Busemann difference profile
$\Delta(x,t) = W_{\theta_2}(x,t;0,0) - W_{\theta_1}(x,t;0,0)$
by Lemma \ref{lem:deltareconstruct}.
Therefore, by Lemma \ref{lem:relativeD}, we can determine $\mathcal{D}_{\theta_2}^{\theta_1,\theta_2}$. 
This determines the geodesic tree $\mathrm{GT}_{\theta}$ by Remark \ref{rem:DandLlength12}. 

By Lemma \ref{lem:deltareconstruct} again, we can reconstruct the Busemann difference profile
 $\Delta_2(x,t) = W_{\theta_2}(x,t;0,0) - W_{\theta}(x,t;0,0)$,
 and also
$$\mathcal{D}_{\theta}^{\theta_1,\theta_2}(p;q)=\mathcal{D}_{\theta_2}^{\theta_1,\theta_2}(p;q)+\Delta_2(p)-\Delta_2(q).$$

Next, as $-\theta_1(n), \theta_2(n) \to\infty$ along monotone sequences, the restricted differential distance $\mathcal D_{\theta}^{\theta_1(n),\theta_2(n)}(p;q)$ is non-increasing by Corollary \ref{cor:theta}, and hence converges to a limit $\mathcal{D}_\theta^{\infty}$.

Let $U$ be the set of point pairs $(p;q)$ for which $p$ is not a forward 3-star and there is a $pq$-geodesic with respect to $\mathcal L$ that is also a $\theta_1-\theta_2$ path for some $\theta_1 \leq \theta_2$. 

By Corollary \ref{cor:theta}, for every $(p;q)\in U$ and for all large enough $n$,
$\mathcal D_{\theta}^{\theta_1(n),\theta_2(n)}(p;q)=\mathcal D_\theta(p;q)$. So $\mathcal{D}_{\theta}^{\infty}(p;q) = \mathcal{D}_{\theta}(p;q)$.

The set $U$ is dense in $\mathbb R^4_\uparrow$ by Proposition \ref{prop:density}. Since $\mathcal D_{\theta}$ is continuous, we see that $\mathcal{D}_{\theta}$ is determined by $\mathcal{D}^{\infty}_{\theta}$:
$$
\mathcal D_\theta(u)=\liminf_{v\to u, v \in U}\mathcal D_\theta^\infty(v).
$$
By Lemma \ref{lem:DandGT}, $\mathcal D_{\theta}$ determines $\mathcal L$.
As such, we have constructed (implicitly, through these steps) a measurable function $F$ that satisfies \eqref{eqn:F} and hence proves the claim of Theorem \ref{thm:main}.

\end{proof}

\section*{Acknowledgements}
The authors thank the hospitality of the Institut Mittag-Leffler during the program ``Random Matrices and Scaling Limits" where this work was partly completed. The authors also thank Evan Sorensen for pointing out some references in the literature. BV was partially supported by an NSERC Discovery Grant. 

\bibliographystyle{dcu} 
\bibliography{geodesictree}
\bigskip

\noindent Mustazee Rahman. Department of Mathematical Sciences, Durham University, Durham, UK. mustazee@gmail.com

\medskip
\noindent B\'alint Vir\'ag. Departments of Mathematics, University of Toronto, 40 St. George St., Toronto, Ontario, M5S2E4. balint@math.toronto.edu
\end{document}